\documentclass[reqno]{article}
\usepackage[utf8]{inputenc}
\usepackage[numbers]{natbib}
\usepackage{amssymb,amsthm}
\usepackage{xcolor}
\usepackage{hyperref}
\usepackage{amsxtra}
\usepackage{mathtools}
\usepackage{txfonts}
\mathtoolsset{showonlyrefs=true}
\usepackage{booktabs}
\usepackage{todonotes} 

\usepackage{algorithm}
\usepackage{algpseudocode}

\newcommand{\Mm}{\mathbb{M}} 
\newcommand{\Cc}{\mathbb{C}}
\newcommand{\Rr}{\mathbb{R}}
\newcommand{\Ss}{\mathbb{S}}
\newcommand{\SU}{\mathfrak{su}}

\newcommand{\UU}{\mathfrak{u}}
\newcommand{\LA}{\mathcal{L}}
\newcommand{\GG}{\mathfrak{g}}
\newcommand{\stab}{\mbox{stab}}
\newcommand{\Tr}{\mbox{Tr}}

\theoremstyle{thmstyleone}%
\newtheorem{theorem}{Theorem}%  meant for continuous numbers
\newtheorem{proposition}[theorem]{Proposition}% 

\theoremstyle{thmstyletwo}%
\newtheorem{remark}{Remark}%

\theoremstyle{thmstylethree}%

\title{Solving cubic matrix equations arising in conservative dynamics}
\author{Michele Benzi\footnote{Scuola Normale Superiore, Piazza dei Cavalieri, 7, Pisa, 56126, Italy -  \textit{michele.benzi@sns.it}} \and Milo Viviani\footnote{CRM Ennio De Giorgi - Collegio Puteano, Scuola Normale Superiore, Piazza dei Cavalieri, 3, Pisa, 56126, Italy - \textit{milo.viviani@sns.it}}}
\date{}

\begin{document}


\maketitle

\begin{center}
\textit{Dedicated to Alfio Quarteroni on his 70th Birthday}
\end{center}


\abstract{In this paper we consider the spatial semi-discretization of conservative PDEs. 
Such finite dimensional approximations of infinite
dimensional dynamical systems can be described as flows in suitable matrix 
spaces, which in turn leads to the need to solve polynomial matrix equations, 
a classical and important topic both in theoretical and in applied mathematics.
Solving numerically these equations is challenging due to the presence of several 
conservation laws which our finite models incorporate and which must be retained while integrating the equations of motion. 
In the last thirty years, the theory of geometric integration has provided a variety of techniques to tackle this problem.
These numerical methods require solving both direct and inverse problems in matrix spaces.
We present three algorithms to solve a cubic matrix equation arising in the geometric integration of isospectral flows.
This type of ODEs includes finite models of ideal hydrodynamics, plasma dynamics, and spin particles, which we use as test problems for our algorithms.}



\maketitle

\section{Introduction}
The numerical solution of polynomial matrix equations is a well studied and active field of research \cite{BinIanMei2012,GohLanRod2006,Sim2016}.
Its relevance clearly goes beyond pure mathematics and the development of efficient	algorithms is crucial in several areas of computational science.
Linear matrix equations have been studied since the 19th Century, beginning with Sylvester \cite{Syl1884}.
Linear matrix equations have a variety of different formulations, and according to the specific structure, different techniques can be used to solve them \cite{GohLanRod2006,Sim2016}.
One degree higher, we find quadratic matrix equations.
For such problems the theory is more intricate and various numerical issues may appear \cite{BinIanMei2012}.
Among the quadratic matrix equations, one of the most studied is the continuous-time algebraic Riccati equation (CARE) \cite{BinIanMei2012}, which will appear in Section~\ref{sec:Num_sch}.

In this paper, we study the numerical solution of the cubic matrix equation
\begin{equation}\label{eq:matrix_cubic}
(I - h\LA X)X(I + h \LA X) = Y,
\end{equation}
where $X$ is unknown and $Y$ is given in $\Mm(N,\Rr)$ or $\Mm(N,\Cc)$ (the spaces of $N\times N$ real or complex matrices), $\LA\neq 0 $ is a linear operator acting on matrices and $h\geq 0$.
Equation \eqref{eq:matrix_cubic} appears in the geometric integration of matrix flows of the form:
\begin{equation}\label{eq:isospectral}
\begin{array}{ll}
\dot{Y} &= [\LA Y,Y],\\
Y(0)&=Y_0,
\end{array}
\end{equation}
where $Y=Y(t)$ is a curve in some subspace of the space $\Mm(N,\Rr)$ or $\Mm(N,\Cc)$ and the square brackets denote the commutator of two matrices: $[A,B] = AB - BA$.
The flow of \eqref{eq:isospectral} is isospectral, which means that the eigenvalues of $Y(t)$ do not depend on $t$.
Furthermore, when $\LA $ is self-adjoint with respect to the pairing $\langle A,B\rangle=\Tr(AB)$, \eqref{eq:isospectral} is Hamiltonian (another term is ``Lie-Poisson"), with Hamiltonian function given by 
\begin{equation}\label{eq:hamiltonian}
H(Y)=\frac{1}{2} \Tr(Y\LA Y).
\end{equation}
A discrete approximation of the solution of \eqref{eq:isospectral} is determined, for $h>0$ sufficiently small, by the implicit-explicit iteration defined in \cite{Viv2019} as:
\begin{equation}\label{eq:LiePoisson_int}
\begin{array}{ll}
(I - h\LA X_n)X_n(I + h \LA X_n) &:= Y_n\\
(I + h\LA X_n)X_n(I - h \LA X_n) &=: Y_{n+1}.
\end{array}
\end{equation}
In this scheme, $Y_n$ denotes the approximate solution at time $t_n$.
This scheme preserves the spectrum of $Y_0$ and nearly conserves the Hamiltonian \eqref{eq:hamiltonian}, indeed it is a Lie--Poisson integrator (see \cite{Viv2019}).
We observe that whenever $Y$ belongs to a quadratic matrix Lie algebra 
\[\GG_J=\lbrace Y\in\Mm(N,\Cc)\mid YJ+JY^*=0\rbrace
\]
for some fixed $J\in\Mm(N,\Cc)$, and $\LA :\GG_J\rightarrow\GG_J$, equation \eqref{eq:matrix_cubic} admits solutions in $\GG_J$.
Indeed, the left-hand side of equation \eqref{eq:matrix_cubic} is the differential of the inverse of the Cayley transform, which is know to preserve quadratic Lie algebras \cite{hlw}.
Of particular interest for applications to PDEs is the case of $J=I$, for which $\GG_J=\UU(N)$, the Lie algebra of the skew-Hermitian matrices.
In section~\ref{sec:Num_exp}, we consider the Lie algebra $\SU(N)$, which consists of skew-Hermitian matrices with zero trace.
\begin{remark}
It is not hard to check that the following equalities hold:
\begin{equation}\label{eq:LiePoisson_int_approx}
\begin{array}{ll}
Y_{n+1} &= Y_n + h\left[\LA \left(\dfrac{Y_{n+1}+Y_n}{2}\right),\dfrac{Y_{n+1}+Y_n}{2}\right] + \mathcal{O}(h^2)\\
X_{n+1} &= X_n + \dfrac{h}{2}([\LA X_{n+1},X_{n+1}] + [\LA X_n,X_n]) + \mathcal{O}(h^2).
\end{array}
\end{equation}
Hence, up to a term of order $\mathcal{O}(h^2)$, $Y_n$ evolves via the midpoint scheme, whereas $X_n$ via the trapezoidal scheme.
It is known that the midpoint and the trapezoidal method are \textit{conjugate symplectic} \cite{hlw}.
Hence, we have that $X_n$ evolves accordingly to a scheme $\phi_h^T:X_n\mapsto X_{n+1}$ which is \textit{conjugate isospectral} to the scheme $\phi_h^M:Y_n\mapsto Y_{n+1}$ as defined in \eqref{eq:LiePoisson_int}, i.e. there exists an invertible map $\chi_h$ such that:
\[
\phi_h^T=\chi_h^{-1}\circ\phi_h^M\circ\chi_h.
\]
The map is clearly given by the left-hand side of the first equation in  \eqref{eq:LiePoisson_int}, which we denote as $\chi_h=\phi_{h/2}^{EE}$.
We observe that $\chi_h=I + \mathcal{O}(h)$.
Hence, if $Y_n$ evolves on a compact set, then $X_n$ evolves on a compact set for any $n\geq 0$ \cite[VI.8]{hlw}.
We illustrate this relationship in the figure~\ref{fig:comm_diag} below. 

\begin{figure}[h!]\label{fig:comm_diag}
\includegraphics[scale=.3]{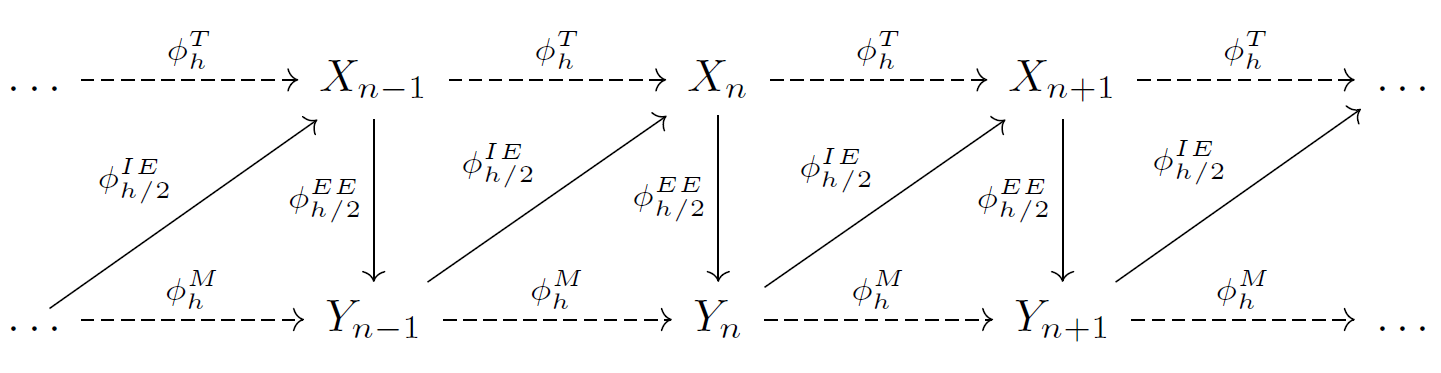}
\caption{Illustration of the schemes \eqref{eq:LiePoisson_int}, \eqref{eq:LiePoisson_int_approx}. Here $\phi_h^M$ denotes the map in the first line of \eqref{eq:LiePoisson_int_approx}, $\phi_h^T$ the map in the second line of \eqref{eq:LiePoisson_int_approx}, $\phi_{h/2}^{IE}$ the map in the first line of \eqref{eq:LiePoisson_int} and , $\phi_{h/2}^{EE}$ the map in the second line of \eqref{eq:LiePoisson_int}. Those correspond up to $\mathcal{O}(h^2)$ terms respectively to the implicit midpoint, trapezoidal, implicit Euler, explicit Euler schemes.}
\end{figure}
\end{remark}

The need for an efficient solver for \eqref{eq:matrix_cubic} can be understood by the fact that conservative PDEs, like the vorticity equation of fluid dynamics \cite{Zei1991,Zei2004} or the drift-Alfv\'en model for a quasineutral plasma \cite{MenBerKamSch2012}, admit a spatial discretization in $\SU(N)$, for $N=1,2,\ldots$.
The crucial aspect of these finite-dimensional models is that they retain the conservation laws of the original equations.
In order to retain these features, the resulting semi-discretized equations can be integrated in time using the scheme \eqref{eq:LiePoisson_int}.
Clearly, to get good spatial accuracy, $N$ has to be quite large (at least $10^3$).
Moreover, the need for an efficient solver for \eqref{eq:matrix_cubic} can be necessary for spin systems with many interacting particles.
In this case, the equations \eqref{eq:isospectral} are posed in the product Lie algebra $\SU(2)^N$, where $N$ is the number of particles.
We stress that in a typical simulation, hundreds of these cubic matrix equations have to be solved to high accuracy.
This paper is devoted to devise an efficient way to solve \eqref{eq:matrix_cubic}.
First we prove existence and uniqueness of the solution for \eqref{eq:matrix_cubic}, for $h$ sufficiently small.
Then, we propose and investigate three possible algorithms to solve equation \eqref{eq:matrix_cubic}, which intrinsically preserve the quadratic matrix Lie algebras. 
First, in section~\ref{sec:exp_fix} we consider an explicit fixed point iteration scheme, the convergence of  which follows from the existence and uniqueness result.
Then, in section~\ref{sec:lin_sch}, we consider a linear scheme again based on a fixed point iteration which requires the solution of a linear matrix equation.
Again the existence and uniqueness result guarantees the convergence of the scheme.
We will see in section~\ref{sec:cubic_scheme} that a suitable inexact Newton method applied to \eqref{eq:matrix_cubic} is also convergent, at least locally.
We will see in section~\ref{sec:quad_scheme} that the third scheme, based on the Riccati equation, is not practical, due to the non uniqueness of the solution.
In the last section, we show the results of numerical experiments aimed at assessing the efficiency of the different schemes for various linear operators $\LA $, which correspond to different physical models.

We mention in passing that a different cubic matrix equation has been studied in \cite{BaMeNeVaD2020}.

\section{Existence and uniqueness}
In this section, we show the existence and uniqueness of the solution for the equation \eqref{eq:matrix_cubic}, when the time-step $h$ is sufficiently small.
\begin{theorem}\label{thm:existence_uniqueness}
Given $Y\in \Mm(N,\Cc)$, equation \eqref{eq:matrix_cubic} has a unique solution for sufficiently small $h>0$ in some neighbourhood of $Y$.
Furthermore, when equation \eqref{eq:matrix_cubic} takes place in $\SU(N)$, the solution is unique in some neighbourhood of $Y\neq 0$ for any $h< \dfrac{1}{3\|\LA\|_{op}\|Y\|}.$
\end{theorem}
\begin{proof}
We rewrite equation \eqref{eq:matrix_cubic} as a fixed point problem $F_h(X)=X$, for 
\begin{equation}\label{eq:FX}
F_h(X) = Y + h[\LA X,X] + h^2(\LA X) X (\LA X).
\end{equation}
We first show the existence of a solution for the fixed point problem $F_h(X)=X$.
Let us introduce the following Cauchy problem:
\begin{equation}\label{eq:diff_eqn1}
\begin{array}{ll}
X'(h) &= G_h^{-1}(X(h))\left[\frac{\partial F_h}{\partial h}(X(h))\right]\\
X(0) &= Y,
\end{array}
\end{equation}
where
\[G_h(X) = I -\frac{\partial F_h}{\partial X},\]
which is invertible for $h$ sufficiently small, being $\|\frac{\partial F_h}{\partial X}\|_{op}=\mathcal{O}(h)$.
Hence, the Cauchy problem \eqref{eq:diff_eqn1} has solution $X(h)$ for $\|\frac{\partial F_h}{\partial X}\|_{op}<1$, being the right hand side continuous.
This ensures the existence of a solution for the fixed point problem $F_h(X)=X$, because of the equality
\[
\frac{d}{dh}[F_h(X(h))]=\frac{\partial F_h}{\partial h}(X(h))+\frac{\partial F_h}{\partial X}[X'(h)].
\]
In order to get uniqueness for the fixed point problem $F_h(X)=X$, we show that there exists a neighbourhood of a fixed point $X$ containing $Y$ in which $F_h$ is a contraction.
Let us calculate
\begin{equation}
\begin{array}{ll}
F_h(X) - F_h(Z) &= h([\LA X,X]-[\LA Z,Z]) + h^2((\LA X) X (\LA X)-(\LA Z) Z (\LA Z))\\
&=h([\LA( X-Z),X]+[\LA Z,X-Z]) +\\
& +h^2((\LA X) (X-Z) (\LA X)+(\LA (X-Z)) Z (\LA X)+(\LA Z) Z (\LA (X-Z))).\\
\end{array}
\end{equation}
Hence
\begin{equation}
\begin{array}{ll}
\|F_h(X) - F_h(Z)\| &\leq h\|\LA\|_{op}(\|X\|+\|Z\|)\|X-Z\|+\\
& +h^2\|\LA\|_{op}^2(\| X\|^2+\|X\|\|Z\|+\|Z\|^2)\|X-Z\|.
\end{array}
\end{equation}
Therefore, given a fixed point $X$ and $0<\varepsilon<1$, in the neighbourhood of $(0,X)$
\begin{equation}\label{eq:areaA}
\begin{array}{ll}
\mathfrak{U}_\varepsilon=(0,X)+&\lbrace (h,Z)\vert\hspace{.3cm} h\|\LA\|_{op}(\|X\|+\|X+Z\|)+ \\
&+h^2\|\LA\|_{op}^2(\| X\|^2+\|X\|\|X+Z\|+\|X+Z\|^2)<1-\varepsilon\rbrace,
\end{array}
\end{equation}
we find $\lbrace h\rbrace\times B_r(X)\subset\mathfrak{U}_\varepsilon$, with $h,r>0$ determined by \eqref{eq:areaA}, such that $F_h:B_r(X)\rightarrow B_r(X)$ is a contraction and in $\overline{B_r(X)}$ we can apply the Banach-Caccioppoli theorem and get a unique solution to the fixed point problem $F_h(X)=X$.
Taking $X+Z=Y$, we get the neighbourhood of $Y$ in which we have a unique solution for equation \eqref{eq:matrix_cubic}.
\\
Let us now assume that equation \eqref{eq:matrix_cubic} takes place in $\SU(N)$.
Then
\begin{equation}
\begin{array}{ll}
\|Y\|^2&=\|(I - h\LA X)X(I + h \LA X)\|^2\\
&=\Tr((I - h\LA X)X(I + h \LA X)((I - h\LA X)X(I + h \LA X))^*)\\
&=\Tr(X(I - (h\LA X)^2)(I - (h\LA X)^2)X^*).
\end{array}
\end{equation} 
The matrix $(I - (h\LA X)^2)(I - (h\LA X)^2)$ is symmetric positive definite for any $X$.
Indeed, the eigenvalues of $-(h\LA X)^2$ are real non negative, since $\LA X\in\SU(N)$.
Therefore, $\|X\|\leq\|Y\|$.
Hence, replacing $\|X\|,\|X+Z\|$ with $\|Y\|$ in \eqref{eq:areaA} and letting $\varepsilon\rightarrow 0$, we can find, for $Y\neq 0$, the following bound for $h$:
\[
h<\dfrac{1}{3\|\LA\|_{op}\|Y\|}.
\]
\end{proof}

\section{Numerical schemes}\label{sec:Num_sch}
In this section, we present three possible numerical schemes to solve \eqref{eq:matrix_cubic}.

\subsection{Explicit fixed point}\label{sec:exp_fix}
Theorem \ref{thm:existence_uniqueness} gives a first scheme to solve equation \eqref{eq:matrix_cubic}:
\begin{equation}\label{eq:scheme_1}
X_{k+1} := F_h(X_k) = Y+ h [\LA X_k,X_k] + h^2\LA X_kX_k\LA X_k
\end{equation}
for $k=1,2,\ldots$.
From theorem \ref{thm:existence_uniqueness} we have that $F_h$ is a contraction for $h$ sufficiently small. 
Hence, the fixed point iteration has a unique solution for $h$ small.
\
When we seek for solutions in $\SU(N)$, we can take any $h<\frac{1}{3\|\LA\|_{op}\|Y\|}$.
The resulting Algorithm 1 is given below. We observe that the  cost per iteration is ${\mathcal O}(N^3)$.
\begin{algorithm}
\caption{Linear scheme in $\SU(N)$}\label{alg:cap}
\begin{algorithmic}
\Require $Y\in\SU(N)$; $\LA:\SU(N)\rightarrow\SU(N)$
\State $tol\gets 10^{-10}$
\State $err \gets 1$
\State $X_0 \gets Y$
\While{$err > tol$}
    \State $X_1 \gets F_h(X_0)$
    \State $err \gets \|X_1-X_0\|$
    \State $X_0 \gets X_1$
\EndWhile
\end{algorithmic}
\end{algorithm}

\subsection{Linear scheme}\label{sec:lin_sch}
Equation \eqref{eq:matrix_cubic} can be decomposed into two coupled matrix equations.
This splitting induces the following scheme:
\begin{equation}\label{eq:scheme_1}
\begin{array}{ll}
&P_k := \LA X_k\\
&(I - hP_k)X_{k+1}(I + h P_k) := Y,
\end{array}
\end{equation}
for $k=1,2,\ldots$.
Note that the second matrix equation in \eqref{eq:scheme_1} is linear in the unknown matrix $X_{k+1}$. 
It is straightforward to check that $X_{k+1}=S(X_k)$, for $S_h(X)=(I - h\LA X)^{-1}Y(I + h\LA X)^{-1}$.
Theorem \ref{thm:existence_uniqueness} gives existence and uniqueness of a solution $\overline{X}$ for $h$ sufficiently small for the equation \eqref{eq:matrix_cubic}.
Hence, since $S_h$ is analytic in the set $\lbrace (h,X)\vert\|h\LA X\|<1 \rbrace$, we can conclude that the fixed point iteration $X_{k+1}=S_h(X_k)$ converges to $\overline{X}$, when $X_0$ is taken in a closed neighbourhood of $\overline{X}$ in which $S_h$ is a contraction.

When $P\in\SU(N)$, we have $(I - hP)^*=(I + h P)$.
Hence, it is enough to calculate the $LU$-factorization for $(I - hP)$ to have the one for $(I + h P)$.
The resulting Algorithm 1 is given below. We observe that the  cost per iteration is ${\mathcal O}(N^3)$.

\begin{algorithm}
\caption{Linear scheme in $\SU(N)$}\label{alg:cap}
\begin{algorithmic}
\Require $Y\in\SU(N)$; $\LA:\SU(N)\rightarrow\SU(N)$
\State $tol\gets 10^{-10}$
\State $err \gets 1$
\State $X_0 \gets Y$
\While{$err > tol$}
	\State $P \gets\LA X_0$
    \State $[L,U] = lu\textunderscore factorization(I-hP)$
    \State $X_1 \gets U^{-1}L^{-1} Y(L^{-1})^*(U^{-1})^*$
    \State $err \gets \|X_1-X_0\|$
    \State $X_0 \gets X_1$
\EndWhile
\end{algorithmic}
\end{algorithm}

\subsection{Quadratic scheme}\label{sec:quad_scheme}
Similarly, we can consider the same decomposition of the previous section for the equation \eqref{eq:matrix_cubic}, 
but reversing the roles of the known and unknown variables.
This splitting induces the following scheme:
\begin{equation}\label{eq:quad_scheme}
\begin{array}{rr}
(I - h P_k)X_k(I + h P_k) &:= Y\\
\LA X_{k+1} &:= P_k,
\end{array}
\end{equation}
where the unknown in the first equation is $P_k$.
The first quadratic equation can be put in the form
\begin{equation}\label{eq:CAREeqn}
h^2PXP + h[P,X] + Y - X = 0,
\end{equation}
which is a type of CARE.

Consider the case when equation \eqref{eq:CAREeqn} is posed in $\SU(N)$.
Then we have two main issues concerning its solvability.
On the one hand, defining $Z:=(I - h P)$, we see that the first equation can be written as:
\begin{equation}\label{eq:quad_eqn}
ZXZ^*=Y.
\end{equation}
Therefore, in order \eqref{eq:quad_eqn} to have solution, $X$ and $Y$ must be congruent.
Hence, $X_0$ must be defined via some congruence transformation of $Y$.
On the other hand, the following proposition shows that \eqref{eq:quad_eqn} admits infinitely many solutions.
Let $I_{p, N-p}$ denote the diagonal matrix with the first $p$ entries equal to 1 and the remaining $N-p$ entries equal to $-1$, where $0\le p \le N$. We denote by $U(p, N-p)$ the Lie group of matrices that leave the bilinear form $b(x,y) = x^* I_{p,N-p}y$ invariant.
\begin{proposition}
Let $A,B\in\SU(N)$ be non-singular, with signature matrix equal to $I_{p,N-p}$. 
Then, the equation
\begin{equation}\label{eq:quad_eqn2}
ZAZ^*=B
\end{equation}
has solution $Z\in GL(N,\Cc)$ (the Lie group of invertible complex matrices) if and only if $Z=C U D$, for some $U\in U(p,N-p)$ and $C,D$ non-singular such that $B = C I_{p,N-p} C^*$ and $DA D^*=  I_{p,N-p}$.  
\end{proposition}
\proof
Let $A,B,C,D$ be as in the hypotheses.
Then, we can rewrite \eqref{eq:quad_eqn2} as
\[
C^{-1}ZD^{-1} I_{p,N-p} (C^{-1}ZD^{-1})^* = I_{p,N-p}.
\]
Therefore, $C^{-1}ZD^{-1}$ is an element of $U(p,N-p)$. 
On the other hand, for any $U\in U(p,N-p)$, we have that $Z = C U D$ is a solution of \eqref{eq:quad_eqn2}.
Hence, all the solutions of \eqref{eq:quad_eqn2} have this form and are parametrized by $U\in U(p,N-p)$.
\endproof
In our particular situation, we are interested in solutions of the form $Z=I + P$, for $P$ skew-Hermitian. 
For instance, taking $A,B\in\SU(N)$ diagonal such that $\mbox{i}(A-B)\geq 0$ and $iA<0$, any $P$ diagonal skew-hermitian (i.e. purely imaginary) such that $P^2=BA^{-1}-I$ is a solution. 
Hence, for generic $A,B$ as above, we get $2^N$ solutions, making the iteration \eqref{eq:quad_scheme} not well-defined.

We can see that the Riccati equation \eqref{eq:CAREeqn} has a non-uniqueness issue also in the following way. 
The equation can be split into two orthogonal components, one parallel to $X$ and one orthogonal to it with respect to the Frobenius inner product:
\begin{equation}
\begin{array}{rr}
h^2\Pi_X(PXP) + \Pi_X Y - X &= 0\\
h^2\Pi_X^\perp(PXP) + h[P,X] + \Pi_X^\perp Y &= 0,
\end{array}
\end{equation}
where $\Pi_X$ is the orthogonal projection onto $\stab_X:=\lbrace A\in\Mm(N,\Cc)\mbox{ s.t. } [A,X]=0\rbrace$.
If we write $P_\parallel:=\Pi_X P$ and $P_\perp:=\Pi^\perp_X P$, we get:
\begin{equation}
\begin{array}{rr}
h^2P_\parallel XP_\parallel + h^2\Pi_X(P_\perp XP_\perp) + \Pi_X Y - X &= 0\\
h^2\Pi_X^\perp(P_\perp X P_\perp) + h^2 P_\perp X P_\parallel + h^2 P_\parallel X P_\perp + h[P_\perp,X] + \Pi_X^\perp Y &= 0.
\end{array}
\end{equation}
If these equations have a solution $(P_\parallel,P_\perp)$, then we also have a solution $(-P_\parallel,P'_\perp)$, where $P'_\perp=P_\perp+\mathcal{O}(h^2)$.
In $\SU(2)\cong\Rr^3$ this is easily seen, since the above scheme reads \citep{Viv2019}:
\begin{equation}
\begin{array}{rr}
h^2p_\parallel (x\cdot p_\parallel) + \Pi_x y - x &= 0\\
h^2p_\perp (x\cdot p_\parallel) + hp_\perp\times x + \Pi_x^\perp y &= 0.
\end{array}
\end{equation}
where $\times$ denotes the vector product and  the matrices in $\SU(2)$ have been represented as vectors in $\Rr^3$, via the standard isomorphism.
Hence, we have the solutions:
\begin{equation}
\begin{array}{rr}
p_\parallel =& \pm\sqrt{\dfrac{\|x\| - \|\Pi_x y\|}{h^2\|x\|}}\dfrac{x}{\|x\|}\\
Rp_\perp =&-\Pi_x^\perp y,
\end{array}
\end{equation}
where $R\in\Mm(3,\Rr)$ is such that $Rp_\perp = h^2p_\perp (x\cdot p_\parallel) + hp_\perp\times x$.
Hence, the ambiguity of the sign in $p_\parallel$ causes a non-uniqueness of solution.

\subsection{Cubic scheme}\label{sec:cubic_scheme}
Newton's method can be directly applied to solve \eqref{eq:matrix_cubic}.
A practical implementation is obtained rewriting \eqref{eq:matrix_cubic} in the following way:
\begin{equation}\label{eq:matrix_cubic2}
F(X)=X - h[\LA X,X] - h^2(\LA X)X(\LA X) - Y=0.
\end{equation}
The Jacobian of $F$ applied to a matrix $Z$ is given by:
\begin{equation}\label{eq:matrix_cubic_jacobian}
\begin{array}{ll}
DF(X)[Z]=&Z - h([\LA Z,X]+[\LA X,Z])\\
& - h^2((\LA Z))X(\LA X)+(\LA X)Z(\LA X)+(\LA X)X(\LA Z)).
\end{array}
\end{equation}
\begin{remark}
Here we consider an inexact Newton approach.
Hence, in order to apply the Newton's method, we consider some approximation for $DF(X)^{-1}$.
We notice that $DF(X)=I - h(\mathcal{B}_1 + h\mathcal{B}_2) $, for $\mathcal{B}_1=[\LA \cdot,X]+[\LA X,\cdot]$ and $\mathcal{B}_2=(\LA \cdot))X(\LA X)+(\LA X)\cdot(\LA X)+(\LA X)X(\LA \cdot)$.
Hence, we get the following third order approximation of $DF(X)^{-1}$
\[
DF(X)^{-1}=I + h\mathcal{B}_1 + h^2(\mathcal{B}_1^2	+ \mathcal{B}_2) + \mathcal{O}(h^3).
\]
At least four reasonable approximations of $DF(X)^{-1}$ can be chosen:
\begin{enumerate}
\item $DF(X)^{-1}\approx I + h\mathcal{B}_1$
\item $DF(X)^{-1}\approx I + h\mathcal{B}_1 + h^2\mathcal{B}_2$
\item $DF(X)^{-1}\approx I + h\mathcal{B}_1 + h^2\mathcal{B}_1^2$
\item $DF(X)^{-1}\approx I + h\mathcal{B}_1 + h^2(\mathcal{B}_1^2+\mathcal{B}_2)$
\end{enumerate}
We have found out that, among those, the second one in general performs better.
Indeed, the first one might have convergence issues for large $h$, and even for large matrices the performances are at most comparable to the second one.
The third one and the fourth one do not perform better than the second one, because the norm of $\mathcal{B}_1^2$ is in general much smaller than the one of $\mathcal{B}_2$.
Hence, the forth one is computationally more expensive than the second one, without any gain in convergence, whereas the third one has a slower convergence than the second one.
\end{remark}
Then, we consider the following approximation for the inverse of the Jacobian evaluated in $F(X)$:
\begin{equation}\label{eq:matrix_cubic_jacobian_approx}
\begin{array}{ll}
DF(X)^{-1}[F(X)]&\approx  \widetilde{DF}(X)[F(X)]\\
&:=F(X) + h([\LA F(X),X]+[\LA X,F(X)])\\
&+ h^2((\LA F(X))X(\LA X)+(\LA X)F(X)(\LA X)+(\LA X)X(\LA F(X))).
\end{array}
\end{equation}
This approximation leads to the inexact Newton scheme (Algorithm 2) described below.
	
\begin{algorithm}
\caption{Inexact Newton scheme}\label{alg:cap}
\begin{algorithmic}
\Require $Y\in\SU(N)$; $\LA:\SU(N)\rightarrow\SU(N)$
\State $tol\gets 10^{-10}$
\State $err \gets 1$
\State $X_0 \gets Y$
\While{$err > tol$}
    \State $X_1 \gets X_0 - \widetilde{DF}(X_0)[F(X_0)]$
    \State $err \gets \|X_1-X_0\|$
    \State $X_0 \gets X_1$
\EndWhile
\end{algorithmic}
\end{algorithm}

We observe that the Inexact Newton method above has its main computational cost in the evaluation of the approximated Jacobian \eqref{eq:matrix_cubic_jacobian}, due to the several matrix-matrix multiplications required.
Hence, for large $N$, the lower complexity of the scheme defined in Section~\ref{sec:lin_sch} makes it more advantageous than the Inexact Newton's one, 
in terms of cost per iteration.
 	
\section{Numerical experiments}\label{sec:Num_exp}
In this section, we test our algorithms on three concrete examples arising from the numerical solution of spatially semi-discretized conservative PDEs:
the incompressible Euler equations, the Drift-Alfv\'en plasma model, and the Heisenberg spin chain.
To integrate in time the equations of motion, we apply the numerical scheme \eqref{eq:LiePoisson_int}.
For each equation, we test the performances of the schemes defined in sections~\ref{sec:lin_sch} and \ref{sec:cubic_scheme}.
Here we briefly summarize our findings.
For large time-step $h$, the scheme of section~\ref{sec:lin_sch} is more efficient when solving \eqref{eq:matrix_cubic} for large matrices. 
This makes the linear scheme more suitable for solving the Euler equations or the Drift-Alfv\'en plasma model.
Analogously, we observe that for spin-systems, for large time-step and many particles, the scheme of section~\ref{sec:lin_sch} is faster and has better convergence properties than the one of section~\ref{sec:cubic_scheme}.

We observe that for both the Euler equations and Drift-Alfv\'en plasma model, the number of iterations tends to decrease with increasing $N$.
This is due to the fact that we are normalizing the initial values of the vorticity and the fact that we are absorbing into the time-step the factor $N^{3/2}$ which should multiply the matrix bracket in order to have spatial convergence of the right-hand side of equations \eqref{eq:Euler_equation_vort_quant} and \eqref{eq:Drift_Alvfen_vort_quant} (see \cite{MoVi2020}).
Indeed, the same phenomenon is not observed for the Heisenberg spin chain, where the spatial discretization is kept fixed while the number of particles is increased.

Another observation is that both in the Euler equations and in the Drift-Alfv\'en plasma model, the small number of Newton iterations for large $N$ prevents any benefit from combining in series two different algorithms, i.e., using a few steps of a fixed point iteration to get a good initial guess for the inexact Newton scheme.
Analogously, we have not observed any improvement in the convergence speed using a mixed scheme for the Heisenberg spin chain.

The simulations are run in Matlab2020a on a Dell laptop, processor Intel(R) Core(TM) i7-1065G7 CPU @ 1.30GHz, RAM 16.0 GB.
For each simulation, a tolerance of $10^{-10}$ has been used as stopping criterion of the iteration.
The results in the tables have been obtained as the average of 10 runs of the respective algorithm for solving equation \eqref{eq:matrix_cubic}, each run with respect to a different randomly generated $Y$.
The CPU time is measured in seconds.

\subsection{2D Euler equations}
The 2D Euler equations on a compact surface $S\subset\Rr^3$ can be expressed in the vorticity formulation as:
\begin{equation}\label{eq:Euler_equation_vort}
\begin{array}{ll}
\dot{\omega}&=\lbrace\psi,\omega\rbrace\\
\Delta\psi &= \omega,
\end{array}
\end{equation}
where $\omega$ is the vorticity field, $\psi$ is the stream function, the curly brackets denote the Poisson brackets, and $\Delta$ is the Laplace-Beltrami operator on $S$.
Let us fix $S=\Ss^2$, the 2-sphere.
Equations \eqref{eq:Euler_equation_vort} admit a spatial discretization called \textit{consistent truncation} (see \cite{Zei1991,Zei2004}), which takes the form:

\begin{equation}\label{eq:Euler_equation_vort_quant}
\begin{array}{ll}
\dot{W}&=[P,W]\\
\Delta_N P &= W,
\end{array}
\end{equation}
where $P,W\in\SU(N)$, for $N=1,2,\ldots$ and for a suitable operator $\Delta_N:\SU(N)\rightarrow\SU(N)$.
Since $\Delta_N$ can be chosen to be invertible, we get that equations \eqref{eq:Euler_equation_vort_quant} are of the form \eqref{eq:isospectral}, with $\LA W=\Delta_N^{-1}W$.
Equations \eqref{eq:Euler_equation_vort_quant} are also a Lie--Poisson system, hence the scheme \eqref{eq:LiePoisson_int} is well-suited to retain its qualitative properties \cite{hlw}.
Clearly, in order to get a good approximation of the equations \eqref{eq:Euler_equation_vort}, we have to take $N$ large (at least around $10^3$, see \cite{MoVi2020}).

In Table~\ref{tab:EE_tab} we show the performances of the different schemes proposed in the previous section.
We notice that for large $N$ the linear scheme performs somewhat better than the inexact Newton one in terms of CPU time. 
\begin{table}[h!]
\begin{tabular}
{|p{2cm}||p{.8cm}|p{1.5cm}||p{.8cm}|p{1.5cm}||p{.8cm}|p{1.5cm}|} \hline & \multicolumn{2}{|c|}{Explicit fixed point}
& \multicolumn{2}{|c|}{Linear scheme} & \multicolumn{2}{|c|}{Newton} \\
 \hline
$N$ & $Iter$ & $CPU time$ & $Iter$ & $CPU time$ & $Iter$ & $CPU time$\\
 \hline
3 & 12 & 0.0008 & 11 & 0.0009 & 6.9 & 0.0020\\
5 & 9.5 & 0.0006 & 8.9 & 0.0005 & 5.5 & 0.0014\\
9 & 9.9 & 0.0006 & 8.3 & 0.0008 & 5.8 & 0.0016\\
17 & 11 & 0.0010 & 8.7 & 0.0021 & 6.1 & 0.0025\\
33 & 7.3 & 0.0022 & 6.5 & 0.0027 & 4.4 & 0.0037\\
65 & 7.8 & 0.0058 & 6.8 & 0.0053 & 4.8 & 0.0090\\
129 & 6.2 & 0.0125 & 5.6 & 0.0151 & 3.9 & 0.0229\\
257 & 7.2 & 0.0447 & 6.4 & 0.0571 & 4.4 & 0.0792\\
513 & 6.3 & 0.2020 & 5.8 & 0.2613 & 4.1 & 0.3836\\
1025 & 7.3 & 1.6191 & 6.3 & 1.9520 & 4.4 & 3.3740
\\
 \hline
\end{tabular}
\caption{Solution of the Euler equation on the sphere. CPU time and number of iterations of the two proposed schemes, for time-step $h=0.5$ and normalized randomly generated initial values.}\label{tab:EE_tab}
\end{table}

\subsection{Drift-Alfv\'en plasma model}
The Drift-Alfv\'en plasma model \cite{MenBerKamSch2012} can be formulated in terms of the so called generalized vorticities $\omega_\pm,\omega_0$. 
Although these are not directly physical quantities, they are a linear combinations of the generalized parallel momentum and the plasma density.
In particular, neglecting the third order non-linearities, and absorbing the parameters in the time variable, we get the following equations:
\begin{equation}\label{eq:Drift_Alvfen_vort}
\begin{array}{ll}
&\dot{\omega}_\pm=\lbrace\Phi_\pm,\omega_\pm\rbrace\\
&\dot{\omega}_0=\lbrace\Phi,\omega_0\rbrace\\
\\
&\Phi_\pm = \Phi \pm \frac{1}{\lambda}\Psi\\
&\Delta \Phi = \omega_+ + \omega_- + \omega_0\\
&\Delta \Psi - \frac{1}{\lambda^2}\Psi =  \frac{1}{\lambda}\omega_+ - \frac{1}{\lambda}\omega_-,
\end{array}
\end{equation}
where $\lambda$ is the ratio between the electron inertial
skin depth and the ion sound gyroradius \cite{MenBerKamSch2012}. 
Analogously to the Euler equations, equations \eqref{eq:Drift_Alvfen_vort} have a matrix representation in $\SU(N)\times\SU(N)\times\SU(N)$, for any $N\geq 1$.
If we assume $\omega_0=0$ (which physically means excluding the electrostatic  drift vortices), we get:
\begin{equation}\label{eq:Drift_Alvfen_vort_w00}
\begin{array}{ll}
&\dot{\omega}_\pm=\lbrace\Phi_\pm,\omega_\pm\rbrace\\
\\
&\Phi_\pm = \Phi \pm \frac{1}{\lambda}\Psi\\
&\Delta \Phi = \omega_+ + \omega_-\\
&\Delta \Psi - \frac{1}{\lambda^2}\Psi =  \frac{1}{\lambda}\omega_+ - \frac{1}{\lambda}\omega_-.
\end{array}
\end{equation}
With the same notation of the previous section, we have the matrix equations:
\begin{equation}\label{eq:Drift_Alvfen_vort_quant}
\begin{array}{ll}
&\dot{W}_\pm=[F_\pm,W_\pm]\\
\\
&F_\pm = F \pm \frac{1}{\lambda}P\\
&\Delta_N F = W_+ + W_-\\
&\Delta_N P - \frac{1}{\lambda^2}P =  \frac{1}{\lambda}W_+ - \frac{1}{\lambda}W_-.
\end{array}
\end{equation}
Equations \eqref{eq:Drift_Alvfen_vort_quant} can be cast in the form \eqref{eq:isospectral} in $\SU(N)\oplus\SU(N)$, for $\LA:\SU(N)\oplus\SU(N)\rightarrow\SU(N)\oplus\SU(N)$, defined by
\[\mathcal{L}(W_+,W_-)=(\Delta_N^{-1}(W_+ + W_-),\frac{1}{\lambda}(\Delta_N - \frac{1}{\lambda^2})^{-1}(W_+ - W_-)).\]
In Table~\ref{tab:DA_tab}, the performances of the three algorithms applied component-wise for the Drift-Alfv\'en model are reported.
Analogously to the Euler equations, the linear scheme performs a bit better than the inexact Newton scheme, especially for large matrices.
\begin{table}[h!]
\begin{tabular}
{|p{2cm}||p{.8cm}|p{1.5cm}||p{.8cm}|p{1.5cm}||p{.8cm}|p{1.5cm}|} \hline & \multicolumn{2}{|c|}{One-step fixed point}
& \multicolumn{2}{|c|}{Two-step fixed point} & \multicolumn{2}{|c|}{Newton} \\
 \hline
$N$ & $Iter$ & $CPU time$ & $Iter$ & $CPU time$ & $Iter$ & $CPU time$\\
\hline
3 & 12 & 0.0013 & 11 & 0.0014 & 6.8 & 0.0024\\
5 & 9.6 & 0.0017 & 8.6 & 0.0020 & 5.6 & 0.0042\\
9 &	11 & 0.0023 & 9.2 & 0.0036 & 6.8 & 0.0045\\
17 & 8.4 & 0.0035 & 7.7 & 0.0040 & 5.1 & 0.0055\\
33 & 6 & 0.0055 & 5.6 & 0.0055 & 3.7 & 0.0081\\
65 & 6.3 & 0.0121 & 5.5 & 0.0157 & 3.9 & 0.0223\\
129 & 6.4 & 0.0502 & 5.9 & 0.0714 & 4 & 0.1019\\
257 & 5.6 & 0.2388 & 5.2 & 0.3207 & 3.7 & 0.4853\\
513 & 5.8 & 1.3479 & 5.2 & 1.6521 & 3.6 & 2.6889\\
1025 & 7 & 10.3985 & 6.1 & 12.2737 & 4.3 & 21.5882\\
 \hline
\end{tabular}
\caption{Solution of the Drift-Alfv\'en model. 
CPU time and number of iterations of the two proposed schemes, for time-step $h=0.5$, $\lambda=5$, and normalized randomly generated initial values.}\label{tab:DA_tab}
\end{table}

\subsection{Heisenberg spin chain}
The Heisenberg spin chain is a conservative model of spin particle dynamics.
This model arises from the spatial discretization of the Landau–Lifshitz–Gilbert Hamiltonian PDE \cite{Lak2011}:
\begin{equation}\label{eq:LL11}
\partial_t\sigma=\sigma\times\partial_{xx}\sigma,
\end{equation}
where $\sigma: \Rr\times \Ss^1\rightarrow\Ss^2$ is a closed smooth curve.
Each value taken by $\sigma$ in $\Ss^2$ represents a spin of an infinitesimal particle.
We notice that unlike the previous examples of hyperbolic PDEs, equation \eqref{eq:LL11} is a parabolic PDE.
In Tables~\ref{tab:HSC_tab} and \ref{tab:HSC_tab2}, we see that this requires a smaller time-step in order to have a comparable number of iterations of the two schemes as for the previous two examples.

Discretizing $\sigma\approx\lbrace s_i\rbrace_{i=1}^{N+1}$ on an evenly spaced grid with step size $\Delta x$ of $\Ss^1$, 
$\lbrace x_i\rbrace_{i=1}^{N+1}$, with the conditions $s_{N+1}=s_1$ and $x_{N+1}=x_1$, we obtain
\[\partial_{xx}\sigma(x_i)\approx \dfrac{s_{i-1} - 2s_i + s_{i+1}}{\Delta x^2}.\]
Each spin vector $s_i$ can be represented by a matrix $S_i$ with unitary norm in $\SU(2)\cong \Rr^3$.
The equations of motion are given by:
\[
\partial_t S_i=\left[S_i,\dfrac{S_{i-1} + S_{i+1}}{\Delta x^2}\right],
\] 
for $i=1,\ldots N$.
Hence, each spin interacts only with its neighbours (which explains the chain name).
Hence, the matrices involved remain very sparse. 
A chain of $N$ particles is an Hamiltonian system in $\SU(2)^N$, with Hamiltonian given by:
\[
H(S_1,\ldots,S_N) = \dfrac{1}{\Delta x^2}\sum_{i=1}^{N} \Tr(S_i^* S_{i+1}).
\]
The operator $\LA: \SU(2)^N\rightarrow \SU(2)^N$ is defined by
\[
\LA(S_1,S_2,\ldots,S_N) = (S_N + S_2,S_1 + S_3,\ldots,S_{N-1}+S_1).
\]
In Tables~\ref{tab:HSC_tab} and \ref{tab:HSC_tab2}, the results for the three algorithms applied component-wise are reported.
In the numerical simulations, we set $\Delta x=1$.
We observe that both the explicit fixed point scheme and the inexact Newton method do not converge for spin-systems with many particles and large time-step $h=0.5$.
On the other hand, the linear scheme does converge for any $N\leq 2^{10}+1$, making it more suitable for long time simulations.
For smaller time-step $h=0.1$ the three algorithms perform almost equally well.

\begin{table}[h!]
\begin{tabular}
{|p{2cm}||p{.8cm}|p{1.5cm}||p{.8cm}|p{1.5cm}||p{.8cm}|p{1.5cm}|} \hline & \multicolumn{2}{|c|}{Explicit fixed point}
& \multicolumn{2}{|c|}{Linear scheme} & \multicolumn{2}{|c|}{Newton} \\
\hline
$N$ & $Iter$ & $CPU time$ & $Iter$ & $CPU time$ & $Iter$ & $CPU time$\\
\hline
3 & 62 & 0.0022 & 24 & 0.0013 & 31 & 0.0041\\
5 & 151 & 0.0039 & 23 & 0.0011 & 71 & 0.0071\\
9 &NC & NC & 24 & 0.0014 & NC & NC\\
17 &NC & NC & 26 & 0.0018 & NC & NC\\
33 &NC & NC& 28 & 0.0028 & NC & NC\\
65 &NC & NC& 31 & 0.0047 & NC & NC\\
129 &NC & NC& 30 & 0.0063 & NC & NC\\
257 &NC & NC& 31 & 0.0121 & NC & NC\\
513 &NC & NC& 32 & 0.0210 & NC & NC\\
1025 &NC & NC& 33 & 0.0386 & NC & NC\\
 \hline
\end{tabular}
\caption{Solution of Heisenberg spin chain model.
CPU time and number of iterations of the two proposed schemes, for time-step $h=0.5$ and normalized randomly generated initial values.}
\label{tab:HSC_tab}
\end{table}

\begin{table}[h!]
\begin{tabular}
{|p{2cm}||p{.8cm}|p{1.5cm}||p{.8cm}|p{1.5cm}||p{.8cm}|p{1.5cm}|} \hline & \multicolumn{2}{|c|}{One-step fixed point}
& \multicolumn{2}{|c|}{Two-step fixed point} & \multicolumn{2}{|c|}{Newton} \\
\hline
$N$ & $Iter$ & $CPU time$ & $Iter$ & $CPU time$ & $Iter$ & $CPU time$\\
 \hline
3 & 10 & 0.0008 & 9.2 & 0.0007 & 6.1 & 0.0014\\
5 & 13 & 0.0009 & 9.7 & 0.0008 & 7.3 & 0.0016\\
9 & 14 & 0.0011 & 9.9 & 0.0010 & 7.7 & 0.0016\\
17 & 14 & 0.0010 & 10 & 0.0013 & 7.6 & 0.0018\\
33 & 14 & 0.0011 & 10 & 0.0018 & 8.1 & 0.0021\\
65 & 15 & 0.0014 & 10 & 0.0019 & 8 & 0.0023\\
129 & 15 & 0.0019 & 10 & 0.0025 & 8 & 0.0031\\
257 & 15 & 0.0025 & 10 & 0.0042 & 8 & 0.0039\\
513 & 15 & 0.0046 & 11 & 0.0076 & 8.1 & 0.0058\\
1025 & 15 & 0.0093 & 11 & 0.0148 & 8.1 & 0.0136\\
 \hline
\end{tabular}
\caption{Solution of Heisenberg spin chain model.
CPU time and number of iterations of the two proposed schemes, for time-step $h=0.1$ and normalized randomly generated initial values.}\label{tab:HSC_tab2}
\end{table}

\section{Conclusions}
In this paper we have proposed and investigated some iterative schemes for the solution of cubic matrix equations
arising in the numerical solution of certain conservative PDEs by means of (Lie--Poisson) geometrical integrators. These types of
schemes enable the preservation of important physical features of the original infinite-dimensional flows, which is generally not the case
when more standard discretizations and numerical integrators are used. 
Both the fixed point iterations and the inexact Newton type scheme we have investigated tend to work well, but we found that the fixed point method which requires the solution of a linear matrix equation is the most robust with respect to the time step and requires a comparable CPU time with the fully explicit scheme.

\section*{Acknowledgements}
The authors would like to thank the anonymous referees for their useful comments and suggestions.
A special thank to prof. Bruno Iannazzo for his observations on the manuscript and for pointing out reference \cite{BaMeNeVaD2020}.


\bibliographystyle{plainnat}
\bibliography{sn-biblio}% common bib file
%\input article_template_format.bbl

%\input sn-sample-bib.tex%

\end{document}